\documentclass[10pt]{article} 
\usepackage{amssymb}
\usepackage{xcolor}
\usepackage{amsmath,amsfonts,amsthm}

\usepackage{comment}

\let\svthefootnote\thefootnote
\newcommand\blankfootnote[1]{%
	\let\thefootnote\relax\footnotetext{#1}%
	\let\thefootnote\svthefootnote%
}
\let\svfootnote\footnote
\renewcommand\footnote[2][?]{%
	\if\relax#1\relax%
	\blankfootnote{#2}%
	\else%
	\if?#1\svfootnote{#2}\else\svfootnote[#1]{#2}\fi%
	\fi
}

\newtheorem{theorem}{Theorem}[section]
\newtheorem{lemma}[subsection]{Lemma}

\newtheorem{corollary}[theorem]{Corollary}
\newtheorem{remark}[theorem]{Remark}
\newtheorem{definition}[theorem]{Definition}

\newtheorem{proposition}{Proposition}

\usepackage{fullpage}

\begin{document}
	
	\numberwithin{equation}{section}

	\title{Mean-stable surfaces in Static Einstein-Maxwell theory}
	
	\author{Fernando Coutinho$^1$ and Benedito Leandro$^2$}
	\footnote[]{$^{1,2}$Universidade Federal de Goiás, IME, CEP 74690-900, Goiânia, GO, Brazil.}
	\footnote[]{$^1$Universidade do Estado do Amazonas, CEST, 1085, CEP 69552-315, Tefé, AM, Brazil.}
	\footnote[]{Email address: \textsf{fcoutinho@uea.edu.br$^1$, bleandroneto@ufg.br$^2$.}}
	\footnote[]{Fernando Soares Coutinho was partially supported by PROPG-CAPES/FAPEAM.}
	
		\footnote[]{Benedito Leandro was partially supported by Brazilian National Council for Scientific and Technological Development (CNPq Grant 403349/2021-4).}

	\date{}

	\maketitle{}
	
	\begin{abstract}
		In this paper we use the theory of mean-stable surfaces (stable minimal surfaces included) to explore the static Einstein-Maxwell space-time. We first prove that the zero set of the lapse function must be contained in the horizon boundary. Then, we explore some implications of it providing some results of nonexistence of stable minimal surfaces in the interior of an electrostatic space, subject to a certain initial boundary data. We finish by proving that the ADM mass is bounded from above by the Hawking quasi-local mass.
	\end{abstract}
	
	\vspace{0.2cm} \noindent \emph{2020 Mathematics Subject
		Classification} : 83C22, 53A10, 83C05.
	
	\vspace{0.4cm}\noindent \emph{Keywords}: Electro-vacuum, constant mean curvature, minimal surface, asymptotically flat.
	

	\section{Introduction}
	\

	The equations of motion for an $(n+1)$-dimensional Einstein-Maxwell space-time are given by
	\begin{eqnarray}\label{1}
	(\widehat{R}ic)_{ij}=2\left(F_{il}F^{l}_{j}-\frac{1}{2(n-1)}F_{lk}F^{lk}\hat{g}_{ij}\right)\quad\mbox{and}\quad d\star F=0; 1\leq i,\,j\leq n+1,
	\end{eqnarray}
	where $F$ represents the electromagnetic field and $\widehat{R}ic$ is the Ricci tensor for the metric $\hat{g}.$ 
	
	Our main ground is the static space-time $(\widehat{M}^{n+1},\hat{g})=M^{n}\times_{f}\mathbb{R}$ such that 
	\begin{eqnarray}\label{0}
	\hat{g}(x,\,t)=g(x)-f^{2}(x)dt^{2};\quad x\in M,
	\end{eqnarray}
	where $(M^{n},g)$ is an open, connected and oriented Riemannian manifold, and $f$ is a smooth warped function. Thus, considering $$F=d\psi\wedge dt,$$
	for some smooth function $\psi$ on $M$, from the warped formulas, \eqref{1} and \eqref{0} we have the following definition (see \cite{chrusciel1999,kunduri2018} and the references therein).
	
	\begin{definition} 
		\label{defA} Let $(M^n,g)$ be an $n$-dimensional smooth Riemannian manifold  with $n\geq 3$ and let $f,\psi:M\rightarrow\mathbb{R}$ be smooth functions satisfying
		\begin{eqnarray}\label{principaleq}
		f{Ric}={\nabla^2}f- \frac{2}{f}\nabla\psi\otimes\nabla\psi + \frac{2}{(n-1)f} |\nabla\psi|^{2}g,
		\end{eqnarray}
		\begin{eqnarray}\label{lapla}
		\Delta f=2\left(\frac{n-2}{n-1}\right)\frac{|\nabla\psi|^{2}}{f}
		\end{eqnarray} 
		and
		\begin{eqnarray*}\label{diver}
			div\left(\frac{\nabla\psi}{f}\right)=0,
		\end{eqnarray*}
		where $Ric$ denotes the Ricci tensor of $(M^n,\, g).$ Moreover, ${\nabla}^{2}$ and $div$ are, respectively, the Hessian and the divergence for $g$ and $\Delta$ is the Laplacian operator. Then $(M^n,\,g,\,f,\,\psi)$ is called an {\it electrostatic system.} The smooth functions $f$, $\psi$ and $M^{n}$ are called \textit{lapse  function}, \textit{electric potential} and spatial factor for the static Einstein-Maxwell space-time, respectively. 
	\end{definition} 
	The above definition implies that the scalar curvature $R$ for the metric $g$ is given by
	\begin{eqnarray}\label{scalarcurv}
	f^{2}R=2|\nabla\psi|^{2}. 
	\end{eqnarray}

	There are some well-known solutions for the electrostatic equations which characterize the static Einstein-Maxwell space-time as a model for an electrically charged black hole. The Reissner-Nordstr\"om electrostatic space-time is one of the most important solutions and represents a model for a static black hole with electric charged $q$ and mass $m$. And it is called subextremal, extremal or superextremal if $m^{2}>q^{2}$, $m^{2}=q^{2}$ and $m^{2}<q^{2}$, respectively. One of the most interesting properties of such electrically charged black hole model is the fact that it allows the existence of multiple black holes in equilibrium due to electric forces involved. Notice that if $\psi$ is a constant function, then the electrostatic system reduces to the static vacuum Einstein equations (cf. \cite{anderson2000,huang2018,miao2003,miao2005,miao2015}), i.e.,
	\begin{eqnarray}\label{vacuum}
	fRic=\nabla^{2}f\quad\mbox{and}\quad\Delta f=0.
	\end{eqnarray}
	These equations characterize the static vacuum space-time, where the most important solution for this system is the Schwarzschild solution. This solution is a model for a static black hole without electric charge.
	
	We now define an important function which was provided by Ruback in \cite{ruback1988} and plays an important role in the study of the static Einstein-Maxwell theory: $$F_{\pm}=f^{2}-(1\pm\psi)^{2}.$$ This function has important properties and it was used before for several authors to get important result about the static Einstein-Maxwell space-time (cf. \cite{chrusciel1999,sophia2019,kunduri2018,ruback1988}). In \cite{kunduri2018}, the authors were able to prove that in an asymptotically flat $3$-dimensional static Einstein-Maxwell space-time $m=|q|$ (extremal electrostatic manifold) if and only if $f=1-\psi$, admitting $f=0$ at $\partial M$. It is worth to say that an extremal Reissner–Nordström space-time contains a unique photon sphere, on which light can get trapped and it has the largest possible
	ratio of charge to mass (cf. \cite{cederbaum2017,sophia2019}). The theory of extremal black holes is very important in physics and has very interesting properties. For instance, extremal charged
	black holes may be quantum mechanically stable, which is consistent with the ideas of cosmic censorship (cf. \cite{horowitz1991}). Moreover, there is evidence that this type of black hole is important to the understanding of the no hair theorem (cf. \cite{burko2021}).
	
	So, in what follows we consider an extremal static Einstein-Maxwell space-time if 
	\begin{eqnarray*}\label{ms}
		f=1-\psi.
	\end{eqnarray*}

	The Reissner-Nordstr\"om manifold can be defined by the Riemannian manifold $M^{n}=\mathbb{S}^{n-1}\times (r^{+},\,+\infty)$ with metric tensor
	$$g=\frac{dr^{2}}{1-2mr^{2-n}+q^{2}r^{2(2-n)}}+r^{2}g_{\mathbb{S}^{n-1}},$$
	where $r$ represents the radius of the Reissner-Nordstr\"om black hole. Here, $m^{2}\geq q^{2}$ are constants, and $r^{+}>(m+\sqrt{m^{2}-q^{2}})^{1/(n-2)}$. The Reissner-Nordstr\"om manifold is one of the most relevant solutions for the electrostatic system, or Einstein-Maxwell equations (Definition \ref{defA}). The scalar curvature $R$ for the Reissner–Nordstr\"om manifold is given by:
	\begin{eqnarray}\label{scalarnordstron}
	R=\dfrac{(n-1)(n-2)q^{2}}{r^{2(n-1)}}.
	\end{eqnarray} So, we can see that the scalar curvature of the Reissner-Nordstr\"om solution is decreasing as long $r$ increases, and goes to zero at infinity. Moreover, the outer horizon for the Reissner-Nordstr\"om space-time is located at $(m+\sqrt{m^{2}-q^{2}})^{1/(n-2)}$, which corresponds to the zero set of the lapse function $(1-2mr^{2-n}+q^{2}r^{2(2-n)})$ of the Reissner-Nordstr\"om manifold. In general, the static horizon is defined as the set where the lapse function for a static manifold is identically zero. This set is physically related with the event horizon, the boundary of a black hole.

	So, an interesting question rises. The study of the zero set of the lapse function in static Einstein-Maxwell theory. This static space-time is more general than the one assumed in \cite{huang2018,miao2005,miao2015}, where the authors considered the static vacuum space-time \eqref{vacuum}. The first static vacuum Einstein black hole model was provided by Schwarzschild. This model is characterized only by the mass of the black hole. The geometry of the zero set of the lapse function for this model is well-known and turns out that it has the geometry of the sphere (see \cite{galloway1993}). The study of the zero set of the lapse function was discussed recently in \cite{huang2018} and \cite{miao2015}. Huang et al. proved that lapse function is zero only at the horizon boundary of the static vacuum space \cite[Theorem 1 ]{huang2018}. This result is closely related with the positive mass theorem \cite[Theorem 7]{huang2018}.

	Inspired by \cite{huang2018,miao2015} we will prove that $f=0$ at the horizon boundary $\partial M$ (Definition \ref{defhorizon}) of an electrostatic manifold $M$ and explore its consequences. Usually, a hypersurface in $M^n$ is called a static horizon if it is the zero level set of the static lapse function $f$, and it is called nondegenerate if the outer normal derivative of the lapse is positive. It is well-known that such horizon is a minimal hypersurface (see \cite[Lemma 20]{sophia2019}). We further remark that the zero set of the lapse function $\{f=0\}$ may be formally defined as the set of limit points of Cauchy sequences on $(M^n,\,g)$ on which $f$ converges to $0$ (cf. \cite{anderson2000}). Let us establish the boundary conditions in the following definition.

	\begin{definition}\label{defhorizon}
		We say that an electrostatic manifold $M^{n}$ has a {\it horizon boundary} if the interior boundary $\partial M$ is a disjoint union of smooth closed and minimal hypersurfaces such that $M$ contains no other closed minimal hypersurfaces. Moreover, we assume that $\partial M$ is locally area minimizing. 
	\end{definition}

	In this work we will generalize some important results about the static vacuum Einstein manifolds contained in \cite{huang2018} and \cite{miao2005} to an electrostatic manifolds (or Einstein-Maxwell manifold) which are more general. Also, we would like to  remark that our results can be useful to extend the results provided in \cite{miao2015} to the static Einstein-Maxwell theory.
	
	Our first result establishes some important properties for an electrostatic manifold having a minimal surface in its interior. In particular, the following theorem provides a local classification for an electrostatic manifold (see the third item of Theorem \ref{propositionsplitting}). Moreover, this result will be crucial towards the next results.

	Without further ado, we state our main results.
	
	\begin{theorem}\label{propositionsplitting}
		Let $(M^3,\,g,\,f,\,\psi)$ be a solution for an electrostatic system.  Let $\Sigma$ be a  locally area minimizing, closed, connected minimal surface.  Suppose that in $\Sigma$ we have $dR\leq0$ and $R\geq3R_\Sigma$, i.e., the scalar curvature $R$ is decreasing and is bounded by the scalar curvature $R_\Sigma$ of $\Sigma$ with respect to the induced metric. Assume that $f$ is not identically zero at $\Sigma$. Moreover, consider the existence of a subset $U$ of $M$ and a diffeomorphism  $\Phi: \Sigma\times [0, \varepsilon) \to U$, where $\Sigma_t:=\Phi(\Sigma\times\{t\})$. Then, the following holds:
		\begin{enumerate}
			\item The volume of $\Sigma_t$ is constant and it is totally geodesic for each $t\in[0,\,\varepsilon)$.
			\item The scalar curvatures $R$ and $R_{\Sigma_t}$ are identically zero in $\Sigma_t$ for each $t$. 
			\item The Ricci curvature of $g$ is zero on $U.$ 
		\end{enumerate}
		Here, $R_{\Sigma_t}$ stands for the scalar curvature of $\Sigma_t$ with respect to the induced metric.
	\end{theorem}
	
	\begin{remark}\label{remarkcurvature}
	To prove the first item of the above theorem no additional assumption over the scalar curvatures are required and no restriction on the dimension is necessary. 
It is well-known that the scalar curvature $R$ of $M$ must be constant at the horizon boundary, which is a minimal hypersurface, for an electrostatic system (cf. \cite{chrusciel1999,sophia2019,kunduri2018,ruback1988}). Therefore, it is reasonable reduce the conditions over the scalar curvatures in Theorem \ref{propositionsplitting} to $R$ constant at $\Sigma$ and $R_{\Sigma}\leq \frac{1}{3}R\Big|_{\Sigma}$, where $R\Big|_{\Sigma}$ is the constant value of $R$ restrict to any minimal hypersurface $\Sigma$ in $M$. Moreover, from \eqref{danada111} and Gauss-Bonnet if $\Sigma$ possess a topology different from the topology of the sphere the restriction  $R_{\Sigma}\leq \frac{1}{3}R$, at $\Sigma$, is not required.
	\end{remark}
	
	\begin{remark}\label{remarkarea}
	    Let us make an experiment regarding the hypothesis $R_\Sigma\leq \frac{1}{3}R$ to see what this condition may represent. Considering an $3$-dimensional Reissner-Nordstr\"om solution we have the outer horizon at $r_0=(m+\sqrt{m^{2}-q^{2}})$, for $m^2\geq q^2$. Thus, from \eqref{scalarnordstron} we can see that $R\leq\frac{2q^2}{r_{0}^{4}}.$ So, the scalar curvature is bounded from above. Assuming, that the scalar curvature for the electrostatic system is bounded from above by a positive constant $c$, i.e., $R\leq c$, we have
	    \begin{eqnarray*}
	  R_\Sigma\leq \frac{1}{3}R\quad\Rightarrow\quad \int_{\Sigma}R_\Sigma\leq \frac{c}{3}|\Sigma|,
	    \end{eqnarray*}
	    where $|\Sigma|$ stands for the area of $\Sigma.$
	    So, from Gauss-Bonnet we have
	    	    \begin{eqnarray*}
2\pi\mathfrak{X}(\Sigma)\leq\frac{c}{3}|\Sigma|.
	    \end{eqnarray*}
	    Here, $\mathfrak{X}(\Sigma)$ represents the Euler characteristic of $\Sigma.$
	    
	    For instance, if $\Sigma$ is a topological sphere we obtain a lower bound for the area of $\Sigma$, i.e.,
	    	    \begin{eqnarray*}
12\frac{\pi}{c}\leq|\Sigma|.
	    \end{eqnarray*}
	    
	   Thus, it is interesting to compare this inequality with the Penrose inequality, which is an upper bound for the area (i.e., an isoperimetric inequality).
	   \end{remark}

	\begin{remark}
		We point it out that in Proposition \ref{lemmaf} we proved that any minimal hypersurface in an $n$-dimensional electrostatic manifold is, in fact, totally geodesic.
	\end{remark}
	
	The following result will generalize Theorem 1 in \cite{huang2018} to an $n$-dimensional Einstein-Maxwell manifold having a horizon boundary, and it is a straightforward consequence of Theorem \ref{propositionsplitting}. It is worth to say that the property ``locally area minimizing'' over the horizon boundary (Definition \ref{defhorizon}) follows directly from the other two properties for $3\leq n\leq 7$.

	\begin{theorem}\label{teo1}
		Let $(M^n,\,g,\,f,\,\psi)$ be a solution for an electrostatic system with a horizon boundary $\partial M$, then $f=0$ in $\partial M$.  
	\end{theorem}
	
	Now, we will show a result which is closely related with the rigidity version of the positive mass theorem. But first, let us establish our asymptotic conditions. Here, an $n$-dimensional electrostatic system is called asymptotic flat if $M^n\backslash K=\displaystyle\cup_{\kappa}E^{n}_{\kappa}$, for some compact subset $K\subset M^n$, and there is a coordinate chart on each end, $E^{n}_{\kappa} = \mathbb{R}^{n}\backslash B_{1}(0)$, such that the the metric $g$ and the lapse function
	$f$ satisfy, respectively,
	\begin{eqnarray*}\label{asymp}
		\quad g_{ij} - \left(1+\dfrac{2m}{r^{n-2}}\right)\delta_{ij}=o(r^{-(n-1)})
		\quad\mbox{and}\quad f - \left(1-\frac{m}{r^{n-2}}\right)=o(r^{-(n-1)}),
	\end{eqnarray*}
	with first derivatives $o(r^{-n})$ as $r \rightarrow\infty$.
	Here, $\delta$ denotes the
	Euclidean metric on $\mathbb{R}^{n}$ and $r :=\sqrt{
		x_1^{2} + \ldots + x_n^{2}}$ denotes the radial coordinate
	corresponding to the coordinates $(x_i)$ on $E_{\kappa}^{n}$. Moreover, $m\in\mathbb{R}$ and it represents the ADM mass.
	
	The following result is closely related with the rigidity part of the positive mass theorem (cf. Theorem 7 in \cite{huang2018}).

	\begin{corollary}\label{coro1}
Let $(M^3,\,g,\,f,\,\psi)$ be an asymptotically flat solution for an electrostatic system with horizon boundary, then there is no end $E^{n}_{k}$ having zero ADM mass. 
	\end{corollary}

	In \cite{huang2018}, considering a static extension the authors were able to prove the nonexistence of closed stable minimal surfaces in a static vacuum manifold subject to a certain boundary data. In what follows, we define the boundary data that ensures the existence of a static extension for a static vacuumm space-time (cf. equation \eqref{vacuum}) and then we generalize \cite[Theorem 12]{huang2018} to an extremal electrostatic manifold. In our proof, the conformal geometry plays an important role.
	\begin{definition}[{\bf Static Extension}]
		Given a Riemannian metric $\gamma$ and a function $H$ on a $2$-sphere, we say that an asymptotically flat three-manifold $(M^{3},\,g)$ with interior boundary $\partial M$ is a static extension subject to the boundary data $(\gamma,\,H)$ if
		\begin{enumerate}
			\item $\partial M$ is diffeomorphic to the $2$-sphere, and the induced metric from $g$ on $\partial M$ is isometric to $\gamma$.
			\item The mean curvature of $\partial M$ with respect to the unit normal vector $\nu$ on $\partial M$ pointing to $(M,\,g)$ is given by $H$.
			\item $(M,\,g)$ admits a bounded lapse function, i.e., the extension is a static manifold
with a bounded static potential.
		\end{enumerate}
	\end{definition}
	
 Since we are considering here electrostatic manifolds $(M^3,\,g)$ satisfying Definition \ref{defA} to avoid any technical problem in our result concerning the existence of a static extension (cf. \cite{miao2003} for more details about static extension), here we will assume the existence of a pair $(\gamma,\,H)$, of a Riemannian metric $\gamma$ and a function $H$, such that 
	$$g\big|_{\partial M}=\gamma\quad\mbox{and}\quad H(\partial M,\,g) = H,$$
	where $H(\partial M,\,g)$ represents the mean curvature of $\partial M$ with respect to $g$.
	
	It is important to point out that for the following results we consider the same conditions over the scalar curvature assumed in Theorem \ref{propositionsplitting} (cf. Remark \ref{remarkcurvature}). Moreover, for the next result assume $g$ analytic on $M.$
	
	\begin{theorem}\label{tabom}
		Let $(M^3,\,g,\,f,\,\psi)$ be a three-dimensional extremal solution for the electrostatic system having a boundary $\partial M$. Suppose the existence of a static extension subject to the boundary data $(\gamma,\,H)$ satisfying $$f^{2}H+2f\frac{\partial{ f}}{\partial\nu} > 0\quad\mbox{and}\quad R(\gamma)-\frac{1}{4}H^{2}\geq3\left(\frac{\partial \log f}{\partial\nu}\right)^{2},$$ where  $R(\gamma)$ and $\nu$ are, respectively, the Gauss curvature of $\gamma$ and the normal vector field of $\partial M$. Then any static extension subject to the boundary data $(\gamma,\,H)$ does not have closed locally area minimizing surfaces.
	\end{theorem}
	
	The proof of the above theorem is closely related to the upper bound estimate for the ADM mass (cf. \cite[Proposition 3]{miao2005}). Therefore, we will now prove such estimate for an extremal three dimensional static Einstein-Maxwell space-time. As a consequence of such estimate, we can conclude a rigidity result for the extremal electrostatic manifold. Moreover, to understand Bartnik's quasi-local mass \cite{bartnik1989} of a boundary surface, it is interesting to estimate the ADM mass of a given static and asymptotically flat manifold in terms of its boundary data $(\gamma,\,H)$.
	
	To do that remember, a constant mean curvature surface (CMC) $\Sigma$ is mean-stable if 
	$$\int_{\Sigma}\{|\nabla\Psi|^{2}-[Ric(\nu,\,\nu)+|A|^{2}]\Psi^{2}\}d\sigma\geq0,$$
	for any function $\Psi$ such that $\int_{\Sigma}\Psi = 0.$ Here, $A$ stands for the second fundamental form of $\Sigma$ with normal vector $\nu$. 
	
	Throughout history, many definitions of quasi-local mass were introduced by different authors (cf. \cite{bartnik1989} and the references therein). Hawking defined a quasi-local mass which has several desirable physical properties (cf. \cite{Christodoulou1986}). This quantity plays an important role in the next theorem. Let us remember that the {\it Hawking quasi-local mass} of a surface $\Sigma$ is given by
	\begin{eqnarray}\label{mh}
	m_{\mathcal{H}}(\Sigma) = \sqrt{\frac{|\Sigma|}{16\pi}}\left(1 - \frac{1}{16\pi}\int_{\Sigma}H^{2}d\sigma\right),
	\end{eqnarray}
	where $|\Sigma|$ is the area of $\Sigma$ and $H$ its mean curvature.

	\begin{theorem}\label{tabom1}
		Let $(M^3,\,g,\,f,\,\psi)$ be a three-dimensional asymptotically flat extremal solution for the electrostatic system. Consider that the interior boundary $\partial M$ is a mean-stable hypersurface such that 
		\begin{eqnarray}\label{condruim}
		f^{2}H +2f\frac{\partial f}{\partial\nu}=\varepsilon\quad\mbox{and}\quad K-\frac{1}{4}H^{2}\geq3\left(\frac{\partial \log f}{\partial\nu}\right)^{2}
		\end{eqnarray}
		where $\varepsilon$ is a positive constant and $K$ is the Gauss curvature of $\partial M$, with normal vector $\nu$ pointing towards the end of the manifold. Then, $$m\leq\frac{4}{\kappa}m_{\mathcal{H}}(\partial M)\sqrt{\dfrac{16\pi}{|\partial M|\varepsilon^{2}}},$$
		where $m$ is the ADM mass, $|\partial M|$ is the area of $\partial M$ and $m_{\mathcal{H}}(\partial M)$ is the Hawking quasi-local mass of $\partial M$. Here, $\kappa>0$ represents the number of ends of the manifold.
	\end{theorem}
	
	\begin{remark}
		We should point out that the condition \eqref{condruim} is important to guarantee that $f>0$ in $M$. This fact will be clarified in the proofs.
	\end{remark}

	\begin{remark}The photon spheres are closely related to constant mean curvature surfaces in the static space-time (cf. \cite[Proposition 2.5]{cederbaum2017} and \cite[Proposition 16]{sophia2019}). We also recommend to the reader to see the work of Virbhadra and Ellis about the photon sphere (cf. \cite{claudel2001} and the references therein). They are regions where light can be confined in closed orbits. Moreover, photon spheres are related to the existence of relativistic images in the context of gravitational lensing. So, Theorem \ref{tabom1} can be useful in the understanding of photon spheres in the extremal static Einstein-Maxwell theory. Here, we can define a photon surface (or photon sphere) as a surface $\Sigma$ with CMC in which $f$ is constant at $\Sigma$.
	
	Now, from Theorem \ref{tabom1} we can provide a rigidity result for the extremal electrostatic manifold. In fact, as we already said before an extremal solution has a unique photon sphere $\Sigma$. Let us say that such photon sphere $\Sigma=\partial M$ is mean-stable with $H=2$ and isometric to the standard unit sphere $\mathbb{S}^{2}\subset\mathbb{R}^{3}$. Under the hypothesis of Theorem \ref{tabom1}, $f$ will be locally a constant function (since $f$ is constant at $\Sigma$ by definition and $\nu(f)=0$ from the hypothesis) and so $(U,\,g)$, $U\subseteq M^3$, is Ricci-flat. Moreover, the Hawking mass \eqref{mh} will be identically zero. Thus, assuming the positive mass theorem holds, we get $m=0=q$ everywhere in $M^3$. Therefore, we must have $(M^3,\,g)$ isometric to Euclidean space with standard metric, by the rigidity part of the positive mass theorem. Thus, the above theorem can lead us towards the classification of electrostatic manifolds (cf. \cite[Corollary 2]{miao2005}).
	\end{remark}

	\section{Background}

	We start this section with a result concerning the existence of stable minimal hypersurfaces in an $n$-dimensional electrostatic manifold. This result is inspired by the one proved in \cite[Lemma 4]{huang2018}. The following result is interesting by itself and it was the inspiration to pursue Theorem \ref{propositionsplitting}. Therefore, we have decided to include its proof here. Moreover, Proposition \ref{lemmaf} and Theorem \ref{propositionsplitting} are closely related to \cite[Theorem 1]{liu2013}.

	\begin{proposition}\label{lemmaf}
		Let $(M^{n},\,g,\,f,\,\psi)$, $n\ge 3$, be an $n$-dimensional solution for the static Einstein-Maxwell equations. Let $\Sigma$ be a  closed, connected, stable minimal hypersurface in $M$. Suppose $f>0$ on $\Sigma$. Then, $\Sigma$ is totally geodesic.
	\end{proposition}
	\begin{proof}
		From the stability of $\Sigma$, for any $\phi\in C^1(\Sigma)$, we have
		\begin{eqnarray*}
			\int_\Sigma |\nabla_\Sigma \phi|^2 \, d\sigma\ge \int_\Sigma\left( |A|^2 + Ric(\nu, \nu)\right) \phi^2\, d\sigma \ge \int_\Sigma Ric(\nu, \nu) \phi^2\, d\sigma, \end{eqnarray*}
		where $\nu$ is the unit normal vector field to $\Sigma$ and $d\sigma$ is the $(n-1)$-volume measure of hypersurfaces. Therefore, since $\Sigma$ is closed and $f$ is a smooth function on $M$ we get
		\begin{eqnarray}\label{operadorsinal1}
		& &\int_\Sigma |\nabla_\Sigma f|^2 \, d\sigma  \ge \int_\Sigma Ric(\nu, \nu) f^2\, d\sigma \nonumber\\
		&\Rightarrow &\int_\Sigma f [\Delta_\Sigma f + Ric(\nu, \nu) f]\, d\sigma \le 0
		\Rightarrow \int_\Sigma f \bigg[\bigg(\Delta_\Sigma + Ric(\nu, \nu)\bigg) f\bigg]\, d\sigma \le 0. 
		\end{eqnarray}

		On the other hand, since $H_\Sigma=0$ by hypothesis, we have (see \cite[p. 60]{galloway1993} and \cite[p. L55]{miao2005})
		\begin{eqnarray*}\label{laplacianos}
			\triangle f=\triangle_\Sigma f+\nabla^2f(\nu,\nu).
		\end{eqnarray*}
		Hence, by Cauchy-Schwarz inequality we  get
		\begin{eqnarray}\label{operadorsinal2}
		&\Rightarrow & \underbrace{\triangle f}_\eqref{lapla}=\triangle_\Sigma f+\underbrace{\nabla^2f(\nu,\nu)}_{\eqref{principaleq}} \nonumber \\
		&\Rightarrow & 2\left(\frac{n-2}{n-1}\right)\frac{|\nabla\psi|^{2}}{f}=\triangle_\Sigma f+fRic(\nu,\nu)+ \frac{2}{f}\langle\nabla\psi,\,\nu\rangle^{2}-\frac{2|\nabla\psi|^{2}}{(n-1)f} \nonumber\\
		&\Rightarrow & (\triangle_\Sigma +Ric(\nu,\nu))f=2f^{-1}(|\nabla\psi|^{2}-\langle\nabla\psi,\,\nu\rangle^{2})\geq 0.
		\end{eqnarray}
		
		From \eqref{operadorsinal1} and \eqref{operadorsinal2} it follows that $(\triangle_\Sigma +Ric(\nu,\nu))f=0.$ Again, from the stability inequality we get
		\begin{eqnarray*}
			&\Rightarrow & \int_\Sigma |\nabla_\Sigma f|^2 \, d\sigma-\int_\Sigma Ric(\nu,\nu) f^2 \, d\sigma \ge \int_\Sigma |A|^2  f^2\, d\sigma \\
			&\Rightarrow & -\int_\Sigma f\triangle_\Sigma f \, d\sigma-\int_\Sigma Ric(\nu,\nu) f^2 \, d\sigma \ge \int_\Sigma |A|^2  f^2\, d\sigma \\
			&\Rightarrow & -\int_\Sigma f\underbrace{\bigg(\triangle_\Sigma f+Ric(\nu,\nu) f\bigg)}_{=0} \, d\sigma \ge \int_\Sigma |A|^2  f^2\, d\sigma \\
			&\Rightarrow & 0 \ge \int_\Sigma |A|^2  f^2\, d\sigma.
		\end{eqnarray*}
		Thus, $|A|=0$ and $\Sigma$ is totally geodesic.

	\end{proof}

	The proof of the next lemma is based in \cite[Proposition 3.2]{Brendle2013}. We extended its ideas to an
	$n$-dimensional electrostatic manifold $(M^n,\,g,\,f,\,\psi)$. Let $\Sigma$ be a two-sided smooth hypersurface in $U\subseteq M^n$ in which $f>0$. Let $\nu$ denote the outward-pointing unit normal to $\Sigma$. Consider the conformally modified metric $\bar{g}=\frac{1}{f^2} g$. Furthermore, we can infer the existence of $\varepsilon>0$ in which $\Phi: \Sigma \times [0,\varepsilon) \rightarrow U$, the normal exponential map with respect to the metric $\bar{g}$, is well-defined and such that $f>0$ in $\Sigma_t = \Phi(\Sigma\times\{t\})$. Precisely, for each point $x\in \Sigma$ the curve $\gamma_x(t)=\Phi(x,t)$ is a geodesic with respect to  $\bar{g}$ such that
	\begin{equation*}
	\Phi(x,0)=x \mbox{ and } \Phi(x,t)=exp_x(-tf(\gamma(t))\nu).
	\end{equation*}
	Moreover,
	\begin{eqnarray}\label{expflow}
	\frac{\partial \Phi}{\partial t}(x,t)\bigg|_{t=0}=-f(x)\nu(x).
	\end{eqnarray}
	Therefore, $\Sigma_t$ is the hypersurface obtained by pushing out along the normal geodesics to $\Sigma$ in the metric $\bar{g}$ a signed distance $t$. Note that the geodesic $\gamma$ has unit speed with respect to $\bar{g}$. Let 
	$H(x, t)$ and $A(x, t)$ be the mean curvature and second fundamental form of $x\in\Sigma_t$ with
	respect to $\nu$ in the metric $g$ (see also \cite[pg. 60]{galloway1993}).
	
	\begin{lemma} \label{lemma:monotonicity}
		Let $(M^{n},\,g,\,f,\,\psi)$, $n\ge 3$, be an $n$-dimensional solution for the static Einstein-Maxwell equations. The mean curvature and second fundamental form of $\Sigma_t$ satisfy the following differential inequality
		\begin{equation*}
		\frac{d}{dt}\left(\frac{H}{f} \right)\ge |A|^2.   
		\end{equation*}
	\end{lemma}
	Moreover, if the mean curvature of $\Sigma_0$ is non negative then $H(\cdot,\,t)\geq0$ for all $t.$
	\begin{proof}
		Since $\Delta_{\Sigma_t} f=\Delta f-\nabla^2 f(\nu,\nu)-H\langle \nabla f,\nu\rangle$, from \eqref{principaleq} and \eqref{lapla} we have 
		\begin{eqnarray*}
			\Delta_{\Sigma_t}f+fRic(\nu,\nu)&=&\Delta f-\nabla^2f(\nu,\nu)-H\langle \nabla f,\nu\rangle+fRic(\nu,\nu) \nonumber\\
			&=&\frac{2(n-2)}{f(n-1)}|\nabla \psi|^2-fRic(\nu,\nu)-\frac{2}{f}\langle \nabla \psi,\nu \rangle^2+\frac{2}{f(n-1)}|\nabla \psi|^2\nonumber\\
			& &-H\langle \nabla f,\nu\rangle+fRic(\nu,\nu)\nonumber\\
			&=&\frac{2}{f}\underbrace{[|\nabla \psi|^2-\langle \nabla \psi, \nu\rangle^2]}_{\ge 0}-H\langle \nabla f,\nu\rangle\ge -H\langle \nabla f,\nu\rangle.
		\end{eqnarray*}
		Hence, the mean curvature of $\Sigma_t$ satisfies the evolution inequality (cf. \cite[Lemma 7.6]{bethuel2006})
		\begin{equation}\label{dercurvmed}
		\frac{\partial H}{\partial t}=f|A|^2+\Delta_{\Sigma_t}f+fRic(\nu,\nu)
		\ge f|A|^2-H\langle \nabla f,\nu\rangle .
		\end{equation}

		Moreover, from \eqref{expflow} we have 
		\begin{equation*}
		\frac{\partial f}{\partial t}=d(f\circ\Phi)(\partial_t)=df(\frac{\partial\Phi}{\partial t})=\langle\nabla f,\,\frac{\partial\Phi}{\partial t}\rangle=-f\langle \nabla f,\nu \rangle.
		\end{equation*}
		Thus,
		\begin{eqnarray*}
			\frac{d}{dt}\bigg(\frac{H}{f}\bigg)&=&-\frac{H}{f^2}\frac{\partial f}{\partial t}+\frac{1}{f}\underbrace{\frac{\partial H}{\partial t}}_{\eqref{dercurvmed}}\ge-\frac{H}{f^2}\frac{\partial f}{\partial t}+\frac{1}{f}\bigg[f|A|^2-H\langle \nabla f,\nu\rangle \bigg]\\
			&=&-\frac{H}{f^2}\frac{\partial f}{\partial t}+|A|^2-\frac{H}{f} \langle \nabla f, \nu \rangle =\frac{H}{f^2}f\langle \nabla f,\nu \rangle+|A|^2-\frac{H}{f} \langle \nabla f, \nu \rangle =|A|^2.
		\end{eqnarray*}
		Hence, $\frac{d}{dt}\bigg(\frac{H}{f}\bigg)\ge |A|^2.$ This implies that
		\begin{equation*}
		\frac{H}{f}(x,t)-\frac{H}{f}(x,0)\geq\int_0^t |A|^2(x,\,s)ds.
		\end{equation*}

		So, if the initial hypersurface $\Sigma$ has non negative mean curvature $H(\cdot,0)$, we conclude that the hypersurface $\Sigma_t$ has non negative mean curvature for each $t$.
	\end{proof}

	\section{Proof of the Main results}

	Before proceeding we must say that Proposition \ref{lemmaf} is very important to understand the first statement of Theorem \ref{propositionsplitting}. The proof of Theorem \ref{propositionsplitting} is extensive, so we have divided the proof in three parts. Thus, the reader will see that the major differences from \cite[Proposition 5]{huang2018} lies in the second part of the proof. 
	
	\begin{proof}[Proof of Theorem \ref{propositionsplitting}]
		
		\
		
		{\bf Proof of the first statement.} Assume $f>0$ on $\Sigma$. Consider the deformation $\Phi:\Sigma\times [0, \varepsilon) \to U$ given by the normal exponential map with respect to the conformally modified metric $f^{-2} g$ in  a collar neighborhood of $\Sigma$ where $f>0$.  Let $\Sigma_t = \Phi(\Sigma\times \{ t\})$ and $\Sigma_0 = \Sigma$. Moreover, $H(\cdot, t)$ and $A(\cdot, t)$ are the mean curvature and second fundamental form of $\Sigma_t$ in  the metric $g$, respectively. Lemma~\ref{lemma:monotonicity} implies that  $H(\cdot, t)\ge 0$ for $t\in (0, \varepsilon)$. From the first variation of area (cf. \cite[p. 60]{galloway1993} and \cite[p. 2649]{huang2018}), we have
		\begin{equation*}
		|\Sigma_t| - |\Sigma_0| = \int_0^t\left( -\int_{\Sigma_s} fH(\cdot, s) \, d\sigma \right) ds,   
		\end{equation*}
		where $|\Sigma_t|$ stands for the area of $\Sigma_t$ for all $t$. For $\varepsilon$ sufficiently small, $\Sigma$ is locally area minimizing. Therefore, the above identity implies that the mean curvature of $\Sigma_t$ cannot be strictly positive for $t<\varepsilon$. Hence $H(\cdot, t)\equiv 0$ for all $t\in[0,\,\varepsilon)$ and the $(n-1)$-volume of $\Sigma_t$ is a constant. Moreover, from Lemma~\ref{lemma:monotonicity}, $A(\cdot, t) \equiv 0$ and  $\Sigma_t$ is totally geodesic for $t\in [0, \varepsilon)$ with respect to the metric $g$. 
		
		\
		
		{\bf Proof of the second statement.}
		Furthermore, since $A(\cdot, t) \equiv 0$, from the first variation of the second fundamental form (cf. \cite[Lemma 7.6]{bethuel2006}), we obtain, for vectors $X, Y$ tangential to $\Sigma_t$,  
		\begin{equation*}
		\nabla_{\Sigma_t}^2 f(X, Y) + Rm(\nu, X, Y, \nu) f=0,   
		\end{equation*}
		where $\nabla_{\Sigma_t}$ denotes the connection  of $\Sigma_t$, $\nu$ is a unit normal vector to $\Sigma_t $ (both with respect to the metric $g$), and $Rm$ is the Riemann curvature tensor  of $(M, g)$ (with the sign convention that the Ricci tensor is the trace on the first and fourth components of $Rm$).   Because $\Sigma_t$ is totally geodesic,  $\nabla_{\Sigma_t}^2 f(X, Y) = \nabla^2 f(X, Y)$ for tangential vectors $X, Y$. Then by the equation \eqref{principaleq} and the assumption that  $f>0$,  we obtain 
		\begin{eqnarray*}
			& &\nabla^2 f=fRic +\frac{2}{f}\nabla \psi \otimes \nabla \psi-\frac{2|\nabla \psi|^2}{f(n-1)}g\\
			&\Rightarrow &\nabla_{\Sigma_t}^2 f=fRic +\frac{2}{f}\nabla \psi \otimes \nabla \psi-\frac{2|\nabla \psi|^2}{f(n-1)}g\\
			&\Rightarrow &- Rm(\nu, X, Y, \nu) f=fRic(X,Y) +\frac{2}{f}\nabla \psi \otimes \nabla \psi(X,Y)-\frac{2|\nabla \psi|^2}{f(n-1)}g(X,Y),\\
		\end{eqnarray*}
		and therefore
		\begin{equation*}Rm(\nu, X, Y, \nu)=-Ric(X,Y) -\frac{2}{f^2}\nabla \psi \otimes \nabla \psi(X,Y)+\frac{2|\nabla \psi|^2}{f^2(n-1)}g(X,Y).\end{equation*} 
		
		Now, for an orthonormal frame $\{ E_i \}$ on $\Sigma_t$, 
		\begin{eqnarray*}
			Ric(X, Y)& =& Rm(\nu, X, Y, \nu) + \displaystyle\sum_{i=1}^{n-1} Rm(E_i, X, Y, E_i)= \\
			&=& -Ric(X,Y) -\frac{2}{f^2}\nabla \psi \otimes \nabla \psi(X,Y)+\frac{2|\nabla \psi|^2}{f^2(n-1)}g(X,Y)+  Ric_{\Sigma_t} (X, Y).
		\end{eqnarray*}	
		It gives that, for all tangential vector fields $X, Y$ to $\Sigma_t$, 
		\begin{equation} \label{equation:Ricci}
		Ric(X, Y) = \frac{1}{2} Ric_{\Sigma_t}(X, Y)-\frac{1}{f^2}\nabla \psi \otimes \nabla \psi(X,Y)+\frac{|\nabla \psi|^2}{f^2(n-1)}g(X,Y).
		\end{equation}
		Then, combining the previous formula with \eqref{principaleq} it gives us
		\begin{eqnarray*}
			& & Ric(X, Y) = \frac{1}{2} Ric_{\Sigma_t}(X, Y)-\frac{1}{f^2}\nabla \psi \otimes \nabla \psi(X,Y)+\frac{|\nabla \psi|^2}{f^2(n-1)}g(X,Y)\\
			&\Rightarrow & \frac{\nabla^2 f}{f}(X,Y)-\frac{2}{f^2}\nabla \psi \otimes \nabla \psi (X,Y)+\frac{2|\nabla \psi|^2}{f^2(n-1)}g(X,Y)= \frac{1}{2} Ric_{\Sigma_t}(X, Y)\\
			& & -\frac{1}{f^2}\nabla \psi \otimes \nabla \psi(X,Y)+\frac{|\nabla \psi|^2}{f^2(n-1)}g(X,Y)\\
			&\Rightarrow & \frac{\nabla^2 f}{f}(X,Y)= \frac{1}{2} Ric_{\Sigma_t}(X, Y)+\frac{1}{f^2}\nabla \psi \otimes \nabla \psi(X,Y)-\frac{|\nabla \psi|^2}{f^2(n-1)}g(X,Y)\\
			&\Rightarrow & \nabla^2 f(X,Y)= \frac{f}{2} Ric_{\Sigma_t}(X, Y)+\frac{1}{f}\nabla \psi \otimes \nabla \psi(X,Y)-\frac{|\nabla \psi|^2}{f(n-1)}g(X,Y).\\
		\end{eqnarray*}
		Thus by \eqref{scalarcurv},
		
		\begin{equation}\label{equation:induced-static-1}
		\nabla_{\Sigma_t}^2 f(X,Y)= \frac{f}{2} Ric_{\Sigma_t}(X, Y)+\frac{1}{f}\nabla \psi \otimes \nabla \psi(X,Y)-\frac{fR}{2(n-1)}g(X,Y).
		\end{equation}
		Taking the trace of \eqref{equation:induced-static-1} and considering $R_{\Sigma_t}$ the scalar curvature of $\Sigma_t$, we get
		\begin{eqnarray}\label{equation:induced-static-2}
		& &\triangle_{\Sigma_t} f = \frac{1}{2}fR_{\Sigma_t}+trace_{\Sigma_{t}}(\frac{1}{f}\nabla \psi \otimes \nabla \psi)-\frac{1}{2}fR \nonumber\\
		& &\underbrace{\triangle_{\Sigma_t} f}_{\ref{laplacianos}} = \frac{1}{2}fR_{\Sigma_t}-\frac{1}{2}fR+\frac{|\nabla_{\Sigma_t}\psi|^{2}}{f}\\
		&\Rightarrow & \underbrace{\triangle f}_\eqref{lapla}-\underbrace{\nabla^2f(\nu,\nu)}_\eqref{principaleq} = \frac{1}{2}fR_{\Sigma_t}-\frac{1}{2}fR+\frac{|\nabla \psi|^2-\langle \nabla \psi, \nu\rangle^2}{f} \nonumber\\
		&\Rightarrow & \frac{2(n-2)}{f(n-1)}|\nabla \psi|^2-fRic(\nu,\nu)-\frac{2}{f}\langle \nabla \psi, \nu\rangle^2+\frac{2}{f(n-1)}|\nabla \psi|^2\nonumber\\
		& &=\frac{1}{2}fR_{\Sigma_t}-\frac{1}{f}\langle\nabla\psi,\,\nu\rangle^{2}-\frac{1}{2}fR+\frac{|\nabla \psi|^2}{f}\nonumber\\
		&\Rightarrow & \frac{|\nabla \psi|^2}{f}-\frac{\langle \nabla \psi, \nu\rangle^2}{f}=\underbrace{\frac{1}{2}fR_{\Sigma_t}+fRic(\nu,\nu)}_{\mbox{Gauss equation}}-\frac{1}{2}fR\nonumber\\
		&\Rightarrow & \frac{|\nabla \psi|^2}{f}-\frac{\langle \nabla \psi, \nu\rangle^2}{f}= 0\nonumber\\
		&\Rightarrow & |\nabla \psi|^2=\langle \nabla \psi, \nu\rangle^2\quad\Rightarrow\quad |\nabla_{\Sigma_t}\psi|^2=0;\quad\forall\, t \in[0,\,\varepsilon).\nonumber
		\end{eqnarray}

		So, from \eqref{equation:induced-static-1}  on $\Sigma_t$ we have
		\begin{eqnarray}\label{vai1}
		\nabla_{\Sigma_t}^2 f(X,Y)&=& \frac{f}{2} Ric_{\Sigma_t}(X, Y)-\frac{fR}{2(n-1)}g(X,Y).
		\end{eqnarray}

		Since $\psi$ must be constant in $\Sigma_t.$ Thus, from \eqref{equation:induced-static-2}  we get
		\begin{eqnarray}\label{legal}
		{\triangle_{\Sigma_t} f} &=& \frac{1}{2}f(R_{\Sigma_t}-R)	  \quad\mbox{in}\quad\Sigma_t,\quad\forall\, t\in[0,\,\varepsilon).
		\end{eqnarray}

		Take the divergence of \eqref{vai1} on $\Sigma_t$ and using the identity $\mathrm{div}_{\Sigma_t} \left(\nabla_{\Sigma_t}^2 f \right)=d(\Delta_{\Sigma_t} f) + Ric_{\Sigma_t}\cdot \nabla_{\Sigma_t} f $ (contracted Bochner formula), where the dot in the last term denotes tensor contraction. Hence, we derive that, on each $\Sigma_t$,
		
		\begin{eqnarray*}\label{bochnercontracted1}
			0&=&d\underbrace{(\Delta_{\Sigma_t} f)}_{\eqref{legal}}+ Ric_{\Sigma_t}\cdot \nabla_{\Sigma_t} f -\mathrm{div}_{\Sigma_t} \underbrace{\left(\nabla_{\Sigma_t}^2 f \right)}_{\eqref{vai1}}\nonumber\\
			&=&d\left(\frac{1}{2}fR_{\Sigma_t}-\frac{1}{2}fR\right)+ Ric_{\Sigma_t}\cdot \nabla_{\Sigma_t} f -\mathrm{div}_{\Sigma_t} \left(\frac{f}{2}Ric_{\Sigma_t} -\frac{fR}{2(n-1)}g  \right)\nonumber\\
			&=&\frac{1}{2}dfR_{\Sigma_t}+\frac{1}{2}fdR_{\Sigma_t}-\frac{1}{2}dfR-\frac{1}{2}fdR+ Ric_{\Sigma_t}\cdot \nabla_{\Sigma_t} f- \frac{f}{2}\underbrace{\mathrm{div}_{\Sigma_t}Ric_{\Sigma_t}}_{=\frac{1}{2}dR_{\Sigma_t}} \nonumber\\ && -\frac{1}{2} Ric_{\Sigma_t}\cdot \nabla_{\Sigma_t} f+\frac{1}{2(n-1)}dfR+\frac{1}{2(n-1)}fdR \nonumber  \\
			&=&\frac{1}{2}dfR_{\Sigma_t}+\frac{1}{4}fdR_{\Sigma_t}+\underbrace{\frac{1}{2} Ric_{\Sigma_t}}_{\eqref{vai1}}\cdot \nabla_{\Sigma_t} f  -\frac{(n-2)}{2(n-1)}dfR-\frac{(n-2)}{2(n-1)}fdR \nonumber \\
			&=&\frac{1}{2}dfR_{\Sigma_t}+\frac{1}{4}fdR_{\Sigma_t}+\left(\frac{\nabla^2_{\Sigma_t}f}{f}+\frac{1}{2(n-1)}Rg_{\Sigma_t}\right)\cdot \nabla_{\Sigma_t}f-\frac{(n-2)}{2(n-1)}dfR-\frac{(n-2)}{2(n-1)}fdR \nonumber \\
			&=&\frac{1}{2}dfR_{\Sigma_t}+\frac{1}{4}fdR_{\Sigma_t}+\frac{\nabla^2_{\Sigma_t}f(\nabla_{\Sigma_t} f)}{f}
			+\frac{1}{2(n-1)}Rdf   -\frac{(n-2)}{2(n-1)}dfR-\frac{(n-2)}{2(n-1)}fdR \nonumber \\
			&=&\frac{1}{2}dfR_{\Sigma_t}+\frac{1}{4}fdR_{\Sigma_t}+\frac{d|\nabla_{\Sigma_t}f|^2}{2f}-\frac{(n-3)}{2(n-1)}Rdf   -\frac{(n-2)}{2(n-1)}fdR. \nonumber\\
		\end{eqnarray*}

Then, a straightforward computation gives
		\begin{eqnarray*}\label{bochnercontracted11}
	0&=&\frac{1}{2}dfR_{\Sigma_t}+\frac{1}{4}fdR_{\Sigma_t}+\frac{d|\nabla_{\Sigma_t}f|^2}{2f}-\frac{(n-3)}{2(n-1)}Rdf   -\frac{(n-2)}{2(n-1)}fdR \nonumber\\
			&=&\frac{1}{4f}\left\{2fdfR_{\Sigma_t}+f^2dR_{\Sigma_t}+2d|\nabla_{\Sigma_t}f|^2-2\frac{(n-3)}{(n-1)}Rfdf-2\frac{(n-2)}{(n-1)}f^2dR\right\} \nonumber \\ 
			&=&\frac{1}{4f}\left\{{d\left[f^2R_{\Sigma_t}+2|\nabla_{\Sigma_t}f|^2\right]-2\frac{(n-3)}{(n-1)}Rfdf}  -2\frac{(n-2)}{(n-1)}f^2dR \right\}. \nonumber\\
		\end{eqnarray*}
		Hereafter, consider $n=3$.	Hence,
		\begin{eqnarray}\label{danada}
		d\left[f^2R_{\Sigma_t}+2|\nabla_{\Sigma_t}f|^2\right] = f^2dR. 
		\end{eqnarray}
		Assuming that $dR\leq0$ on $\Sigma_t$. Then, by integration we get
		\begin{eqnarray*}
			0\geq\int_{\Sigma_t}(f^2dR) d\sigma_t =\int_{\Sigma_t}d\left[f^2R_{\Sigma_t}+2|\nabla_{\Sigma_t}f|^2\right]d\sigma_t = 0 \quad\Rightarrow\quad R\quad\mbox{is constant at}\quad\Sigma_{t}. 
		\end{eqnarray*}
		Therefore, from \eqref{danada} we have
		\begin{eqnarray}\label{magicid}
		f^2R_{\Sigma_t}+2|\nabla_{\Sigma_t}f|^2 = c. 
		\end{eqnarray}
		Furthermore, from $f^2R_{\Sigma_t}  + 2|\nabla_{\Sigma_t} f|^2 =c$ and $\Delta_{\Sigma_t} f=\frac{1}{2}f R_{\Sigma_t}-\frac{1}{2}fR$ we get
		 \begin{eqnarray*}
     2f\Delta_{\Sigma_t} f = f^{2}R_{\Sigma_t}- f^2R &\Rightarrow& 2 f \Delta_{\Sigma_t} f=c- 2|\nabla_{\Sigma_t} f|^2- f^2R \\
    &\Rightarrow& 2 f \Delta_{\Sigma_t} f+ 2|\nabla_{\Sigma_t} f|^2=-f^2R+c. 
    \end{eqnarray*}

		By integration,
		\begin{eqnarray}\label{pqp}
		&\Rightarrow & 2 \int_{\Sigma_t}f \Delta_{\Sigma_t} f d\sigma_t + 2\int_{\Sigma_t}|\nabla_{\Sigma_t} f|^2 d\sigma_t=\int_{\Sigma_t}(-f^{2}R+c)d\sigma_t\nonumber\\
		&\Rightarrow & 0= -2 \int_{\Sigma_t}|\nabla_{\Sigma_t}f|^2 d\sigma_t + 2\int_{\Sigma_t}|\nabla_{\Sigma_t} f|^2 d\sigma_t =\int_{\Sigma_t}(-f^{2}R+c)d\sigma_t\nonumber\\
		&\Rightarrow& R\int_{\Sigma_t}f^{2}d\sigma_t =  c|\Sigma_t|.
		\end{eqnarray}
		This shows us that $R=0$ if and only if $c=0.$

		Suppose $c\neq0$. Using that the Gauss curvature is $2K=R_{\Sigma_t}$ and combining \eqref{legal} with \eqref{vai1} we get
		$$(\nabla^{2}_{\Sigma_t}f)(X,\,Y) = \frac{1}{2}\left(fK-\frac{fR}{2}\right)g(X,\,Y)\quad\Rightarrow\quad\nabla^{2}_{\Sigma_t}f=\frac{\Delta_{\Sigma_t}f}{2}g_{\Sigma_t}.$$
		Since $R$ is constant and $div_{\Sigma_t}(\nabla^{2}_{\Sigma_t}f)=d(\Delta_{\Sigma_t}f)+Kdf$, the divergence of the above equation gives
		\begin{eqnarray*}
			&&0=\frac{1}{2}d(\underbrace{\Delta_{\Sigma_t}f}_{\eqref{legal}}) + Kdf\quad\Rightarrow\quad0=3Kdf+fdK-\frac{R}{2}df\quad\Rightarrow\quad 0= 3Kf^{2}df+f^{3}dK-\frac{Rf^{2}}{2}df\nonumber\\
			&\Rightarrow& 0= d(Kf^{3})-\frac{R}{6}df^{3}\quad\Rightarrow\quad d\left(Kf^{3}-\frac{R}{6}f^{3}\right)=0\quad\Rightarrow\quad d\left(\frac{R_{\Sigma_t}}{2}f^{3}-\frac{R}{6}f^{3}\right)=0.
		\end{eqnarray*}
		So $R_{\Sigma_t}f^{3}-\frac{R}{3}f^{3}=c_1$, for some constant $c_1$ in $\Sigma_t$.
		
		Considering $f>0$ at $\Sigma_t$, from \eqref{legal}
		\begin{eqnarray}\label{danada111}
			&&\int_{\Sigma_t}\left(R_{\Sigma_t}f-\frac{Rf}{3}\right)d\sigma=c_{1}\int_{\Sigma_{t}}f^{-2}d\sigma
			\quad\Rightarrow\quad0<\frac{2R}{3}\int_{\Sigma_t}fd\sigma=c_{1}\underbrace{\int_{\Sigma_{t}}f^{-2}d\sigma}_{>0}\quad\Rightarrow\quad c_1 > 0,
		\end{eqnarray}
		see Remark \ref{remarkarea}.
		Since $R(p)\geq 3R_{\Sigma_t}(p)$ for some $p\in\Sigma_t$, by hypothesis, the only possibility is $c_1=0$, i.e., $3R_{\Sigma_t}=R$. Nonetheless, we have $R$ constant at $\Sigma_t$. Hence, from \eqref{legal} we obtain
		\begin{eqnarray*}
			0=\int_{\Sigma_t}\Delta_{\Sigma_t}f d\sigma=-\frac{1}{3}R\int_{\Sigma_t}f d\sigma,
		\end{eqnarray*}
		Hence, the only possibility is $\int_{\Sigma_t}f d\sigma=0$, which is impossible since $f>0.$ So, $R=R_{\Sigma_t}=0$.
		
		On the other hand, if $R=0$ we get $c=0$. Hence, from \eqref{magicid}  we obtain $f$ constant at $\Sigma_t$. Thus, we conclude our proof following the same steps of the third statement of \cite[Proposition 5]{huang2018}. 	
		
		\
		
		{\bf Proof of the third statement.}
		Furthermore, from \eqref{equation:induced-static-1} and from the fact that
		$R=0$, $f$ and $\psi$ are constant in $\Sigma_t$, we have $Ric_{\Sigma_t}(X,Y)=0$. Consequently, from \eqref{equation:Ricci}, $Ric(X, Y)=Ric_{\Sigma_t}(X,\,Y)=0$ for any tangent vectors in $\Sigma_t$. 
		Since $R=R_{\Sigma_t}=0$ in $\Sigma_t$ for all $t\in[0,\,\varepsilon)$ by  Gaussian equation we have $Ric(\nu,\nu)=0.$

		Moreover, we have $A(\cdot,\,t)=0$ in $\Sigma_t$ for all $t\in[0,\,\varepsilon)$, then by Codazzi equation $$R(X, Y, Z, \nu)=(\nabla_Y B)(X, Z,\nu)-(\nabla_X B)(Y, Z,\nu)$$ we get $R(X, Y, Z, \nu)=0$, and therefore $Ric(X, \nu)=0$, where $\langle B(X,Y),\nu\rangle= A(X,Y)$ in $\Sigma_t$.
		
		Thus, $Ric=0$ in $U\subseteq M$.

	\end{proof}

	Now, the proof for the next theorem is basically the same of \cite[Theorem 1]{huang2018} and follows from the first part of Theorem \ref{propositionsplitting}. We have sketched the proof here for completeness of the text.

	\begin{proof}[Proof of Theorem~\ref{teo1}]
		The proof is by contradiction. If $f$ is not zero on $\partial M$, by Theorem~\ref{propositionsplitting}, there is a subset $U$ of $M$ and diffeomorphism $\Phi :\partial M \times [0,\varepsilon) \rightarrow U$, with $\Phi(\partial M\times {t})=\partial M_t$ and  $\partial M_0=\partial M$ where $H(\cdot,\,t)=0$. That is, a collar neighborhood of $\partial M$ in $M$,  splits as a foliation of minimal hypersurfaces $\partial M_t \subset U\subset M$. It contradicts that $M$ contains no closed minimal hypersurfaces other than $\partial M$. Therefore, $f$ must be zero on $\partial M.$
	\end{proof}

	As consequence of the above theorem, a simple integration of \eqref{lapla} proves the following result. 
	
	\begin{proof}[Proof of Corollary~\ref{coro1}]
		From the fact that any two ends must be separated by a minimal surface and $M$ does not contain any minimal surfaces in its interior other than $\partial M$, $M$ can have only one end with zero ADM mass. From \eqref{lapla} we have
		\begin{eqnarray*}
			&&0\leq2\left(\frac{n-2}{n-1}\right)\int_{M} |\nabla\psi|^{2} dV=\int_{M}f\Delta fdV=-\int_{M}|\nabla f|^{2}dV-\int_{\partial M} f\langle\nabla f,\,\nu\rangle dS+\displaystyle\lim_{r\rightarrow\infty}\int_{S(r)}f\langle\nabla f,\,\eta\rangle dS.
		\end{eqnarray*}
		
		Now, since $f=0$ in $\partial M$ (Theorem \ref{teo1}) and the end has zero ADM mass (i.e., $f=1+o(r^{-(n-1)}$)), we get
		\begin{eqnarray*}
			&&0\leq-\int_{M}|\nabla f|^{2} dV.
		\end{eqnarray*}
		Hence, if $\partial M$ is non-empty, $|\nabla f|=0$ in $M.$ So, $f$ is constant, and $g$ is Ricci-flat by \eqref{principaleq} and \eqref{lapla}. Since any three dimensional Einstein manifold has constant sectional curvature, it must be isometric to Euclidean space. Contradicting the fact that the horizon boundary is a compact minimal surface in $M.$  
	\end{proof}

	The proof of the next two theorems follows some ideas contained in \cite{huang2018,miao2005}, and we also follow closely their notations. Thus, will be more easier to compare the proofs. On the other hand, the conformal geometry was fundamental and in that sense turns our proofs more delicate. 
	\begin{proof}[Proof of Theorem \ref{tabom}]
		Consider $f=1-\psi$ (extremal solution). Now, for the conformal metric $\overline{g}=f^{-2}g$ the laplacian operator $\overline{\Delta}$ of a smooth function $F:M\rightarrow\mathbb{R}$ is given by
		\begin{eqnarray*}
			\overline{\Delta}F&=&f^{2}\left[\Delta F  - \frac{(n-2)}{f}\langle\nabla F,\,\nabla f\rangle\right].
		\end{eqnarray*}
		Considering $n=3$, $F=f$ and $\psi = 1 - f$ we get
		\begin{eqnarray}\label{pdg}
		\overline{\Delta}f = f^{2}\underbrace{\left(\Delta f - \frac{|\nabla f|^{2}}{f} \right)}_{= 0\quad\mbox{by}\quad\eqref{lapla}} = 0.
		\end{eqnarray}
		Consequently,
		\begin{eqnarray*}
			0=\overline{\Delta}f = \overline{\Delta}_{\partial M}f + \overline{H}\frac{\partial f}{\partial\nu} + \overline{\nabla}^{2}f(\nu,\,\nu),
		\end{eqnarray*}
		where $\nu$ and $\overline{H}$ are the normal and the mean curvature with respect to $\bar{g}$, respectively.
		
		For any function $F$ on $M$ the hessian $\overline{\nabla}^{2}$ for the metric $\bar{g}$ is given by
		\begin{eqnarray*}
			\overline{\nabla}^2F(\nu,\nu)&=&f^{2}\left[ \nabla^{2}F(\nu,\nu)-\frac{1}{f}g(\nabla F,\nabla f)g(\nu,\nu)+\frac{2}{f}g\left(\nu, \nabla f\right)g\left(\nabla F,\nu\right)\right].
		\end{eqnarray*}
		Thus, considering $F=f$ we get
		\begin{eqnarray*}
			\overline{\nabla}^2f(\nu,\nu)
			&=&f^{2}\left[\nabla^{2}f(\nu,\nu)-\frac{1}{f}|\nabla f|^{2}+\frac{2}{f}g\left(\nabla f,\,\nu\right)^{2}\right].
		\end{eqnarray*}
		Assuming $\psi=1-f$ and $n=3$, from \eqref{principaleq} we get
		\begin{eqnarray*}
			\overline{\nabla}^2f(\nu,\nu)
			&=&f^{2}\left[fRic(\nu,\nu)-\frac{2}{f}|\nabla f|^{2}+\frac{4}{f}g\left(\nabla f,\,\nu\right)^{2}\right].
		\end{eqnarray*}
		Hence,
		\begin{eqnarray*}
			0= \overline{\Delta}_{\partial M}f + \overline{H}\frac{\partial f}{\partial\nu} + f^{3}Ric(\nu,\,\nu) - 2f|\nabla f|^{2}+4fg\left(\nabla f,\,\nu\right)^{2}.
		\end{eqnarray*}
		
		Since $f$ is harmonic in $(M,\,\bar{g})$, by maximum principle and $f\rightarrow 1$ at infinity, we
		may assume that $\displaystyle\inf_{M}(f)$ occurs on $\partial M$ and $f$ is not a constant function. Let $f(y)= \displaystyle\min_{\partial M}f$, for some $y\in\partial M$. Suppose that $f(y)\leq0$, we can conclude by the maximum principle and Hopf's lemma that $\overline{\Delta}_{\partial M}f(y)\geq0$ and $\frac{\partial f}{\partial\nu}(y)>0$ (cf. \cite[Proposition 3]{miao2005}). So, by Gauss equation we have
		\begin{eqnarray*}
			0= \overline{\Delta}_{\partial M}f + \overline{H}\frac{\partial f}{\partial\nu}+ \frac{1}{2}f^{3}(H^{2}-|A|^{2} - 2R_{\gamma}+R)-2f|\nabla f|^{2}+4fg\left(\nabla f,\,\nu\right)^{2}.
		\end{eqnarray*}
		It is well-known that $\overline{H}=f^{2}H + 2f\frac{\partial f}{\partial\nu}$, then using $f=1-\psi$ and \eqref{scalarcurv} we get
		\begin{eqnarray}\label{verificar}
		-\overline{\Delta}_{\partial M}f - \overline{H}\frac{\partial f}{\partial\nu} =   f\Bigg\{\frac{f^2}{2}(H^{2}-|A|^{2} - 2R_{\gamma})+3\left(\frac{\partial f}{\partial\nu}\right)^{2}+\underbrace{\left(\frac{\partial f}{\partial\nu}\right)^{2}-|\nabla f|^{2}}_{\leq0\quad\mbox{by Cauchy-Schwarz}}\Bigg\}.
		\end{eqnarray}
		Furthermore, since $(n - 1)|A|^2 \geq H^2$, we have $$\frac{f^{2}}{2}(H^{2}-|A|^{2} - 2R_{\gamma})+3\left(\frac{\partial f}{\partial\nu}\right)^{2}\leq \frac{f^{2}}{4}H^{2}-f^{2}R_{\gamma}+3\left(\frac{\partial f}{\partial\nu}\right)^{2}\leq0.$$ Now, from \eqref{verificar} we can see that $f(y)\leq0$ leads us to a contradiction. So, $f(y)=\displaystyle\inf_{M}f>0$ which implies $f>0$ in $M$.

		Suppose, to give a contradiction, that there is a closed, locally  area minimizing surface $\Sigma$ in $M$. By Theorem \ref{propositionsplitting}, $g$ must be Ricci-flat  in an open neighborhood of the minimal surface. Since $ f > 0 $ and $g$ is analytic on $M$, $(M, g)$ 
		has vanishing Ricci curvature. In three dimensions, this implies $(M, g)$ is isometric to an exterior region in the Euclidean space, which is free of closed minimal surfaces. It gives a contradiction.
	\end{proof}
	
	The following proof is based on the above theorem. We will consider the conformal metric $\bar{g}=f^{-2}g$ and use some of the formulas provided in the result proved before. Moreover, it is easy to conclude from the above proof that $f>0$ in $M$ for a mean-stable surface $\Sigma$ satisfying the same conditions over the mean curvature and the Gauss curvature \eqref{condruim}.
	\begin{proof}[Proof of Theorem \ref{tabom1}]
		In what follows, consider $f>0$ (cf. Theorem \ref{tabom}). From \eqref{verificar}, we have
		\begin{eqnarray*}
			\overline{\Delta}_{\partial M}f + \overline{H}\frac{\partial f}{\partial\nu} = \frac{f^3}{2}(2K-H^{2}+|A|^{2})-4f\left(\frac{\partial f}{\partial\nu}\right)^{2}+f|\nabla f|^{2}.
		\end{eqnarray*}
		First, divide the above formula by $f$. Then, using that the volume measure of the surface $\partial M$ is $d\bar{\sigma}=f^{-2}d\sigma$ and $\overline{H}=\varepsilon$, by integration we get
		\begin{eqnarray*}
			\int_{\partial M}\frac{1}{f}\overline{\Delta}_{\partial M}fd\bar{\sigma} &+& \int_{\partial M}\varepsilon f^{-2}\frac{\partial \log f}{\partial\nu}d\sigma = \frac{1}{2}\int_{\partial M}(2K-H^{2}+|A|^{2})d\sigma\nonumber\\
			&-&4\int_{\partial M}\left(\frac{\partial\log f}{\partial\nu}\right)^{2}d\sigma
			+\int_{\partial M}\frac{1}{f^{2}}|\nabla f|^{2}d\sigma.
		\end{eqnarray*}
		
		Now, make $V=\log f$ to obtain
		\begin{eqnarray}\label{laplaV}
		\frac{\Delta f}{f}=\Delta V + |\nabla V|^{2}.
		\end{eqnarray}
		Thus,
		\begin{eqnarray}\label{pes1}
		\int_{\partial M}(\overline{\Delta}_{\partial M}V+\|\overline{\nabla}_{\partial M}V\|^{2})d\bar{\sigma} &+& \varepsilon\int_{\partial M}f^{-2}\frac{\partial V}{\partial\nu}d\sigma = \frac{1}{2}\int_{\partial M}(2K-H^{2}+|A|^{2})d\sigma\nonumber\\
		&-&4\int_{\partial M}\left(\frac{\partial V}{\partial\nu}\right)^{2}d\sigma
		+\int_{\partial M}|\nabla V|^{2}d\sigma.
		\end{eqnarray}
		Here, $\|\cdot\|$ is the norm for the metric $\bar{g}$. Moreover, from \eqref{pdg} and \eqref{laplaV} we can infer that
		\begin{eqnarray*}
			0=\overline{\Delta}V+\|\overline{\nabla}V\|^{2}.
		\end{eqnarray*}

		By integration, the asymptotic conditions give us (cf. \cite[Equation 23]{miao2005})
		\begin{eqnarray*}
			\int_{\partial M}f^{-2}\frac{\partial V}{\partial\nu}d\sigma&=&\sum_{\kappa}\lim_{r\rightarrow\infty}\int_{\mathbb{S}_{\kappa}(r)}f^{-3}\frac{\partial f}{\partial\eta}d\sigma + \int_{M}\|\overline{\nabla} V\|^{2}d\bar{\sigma}\nonumber\\
			&=&
			\kappa\lim_{r\rightarrow\infty}\int_{\mathbb{S}(r)}f^{-3}\langle\nabla f,\,\eta\rangle d\sigma + \int_{M}\|\overline{\nabla} V\|^{2}d\bar{\sigma}\nonumber\\
			&=& m\kappa\displaystyle\lim_{r\rightarrow\infty}\left[\left(\dfrac{\left(1+\dfrac{2m}{r}\right)}{\left(1-\dfrac{m}{r}\right)^{3}}\right)\underbrace{\left(\dfrac{1}{r^{2}}\int_{\mathbb{S}(r)}d\sigma\right)}_{=4\pi}\right] + \int_{M}\|\overline{\nabla} V\|^{2}d\bar{\sigma}\nonumber\\
			&=& 4m\kappa\pi + \int_{M}|{\nabla} V|^{2}d{\sigma},
		\end{eqnarray*}
		where $\nu$ is pointing towards the asymptotic end, and $\|\overline{\nabla}V\|^{2}d\bar{\sigma}=|\nabla V|^{2}d\sigma$. Moreover, $\eta$ stands for the normal vector of the sphere $\mathbb{S}(r)$. Here, $\kappa$ represents the number of ends $E$ of the manifold $M$. Hence,
		\begin{eqnarray}\label{pes2}
		4m\varepsilon\kappa\pi + \varepsilon\int_{M}|{\nabla} V|^{2}d{\sigma} = \varepsilon\int_{\partial M} f^{-2}\frac{\partial V}{\partial\nu}d\sigma.
		\end{eqnarray}
		Furthermore, since $K\geq\frac{1}{4}H^{2}+3\left(\frac{\partial V}{\partial\nu}\right)^{2}$, from the Gauss-Bonnet theorem, \eqref{pes1} and \eqref{pes2} we have
		\begin{eqnarray}\label{mass1}
		\int_{\partial M}|\nabla_{\partial M}V|^{2}d\sigma &+& 4m\kappa\varepsilon\pi+ \varepsilon\int_{M}|\nabla V|^{2}d\sigma \nonumber\\
		&=& \frac{1}{2}\int_{\partial M}(|A|^{2}-H^{2})d\sigma-4\int_{\partial M}\left(\frac{\partial V}{\partial\nu}\right)^{2}d\sigma+\int_{\partial M}|\nabla V|^{2}d\sigma+4\pi.
		\end{eqnarray}

		Now, by the same arguments used by Miao \cite[Proposition 3]{miao2005} we can conclude that if $\partial M$ has constant mean curvature and is mean-stable we have
		\begin{eqnarray}\label{verificar2}
		8\pi\geq \int_{\partial M}\left( |A|^2 + Ric(\nu, \nu)\right)\, d\sigma.
		\end{eqnarray}
		See more details about the above inequality in \cite{Christodoulou1986} and \cite{liyau1982}.
		On the other hand, the Gauss equation gives us
		\begin{eqnarray*}
			|A|^{2}+Ric(\nu,\,\nu)=\frac{1}{2}(H^{2}+|A|^{2} - 2K+R).
		\end{eqnarray*}
		Combining the above identity with  \eqref{scalarcurv} and \eqref{verificar2} provides
		\begin{eqnarray*}
			12\pi-\int_{\partial M}\left(\frac{1}{2}H^{2}+|\nabla V|^{2}\right)\, d\sigma\geq \frac{1}{2}\int_{\partial M}|A|^{2}\, d\sigma.
		\end{eqnarray*}
		
		Then, from \eqref{mass1} we obtain
		\begin{eqnarray*}
			4m\kappa\varepsilon\pi\leq\underbrace{4\int_{\partial M}\left(\frac{\partial V}{\partial\nu}\right)^{2}d\sigma+\int_{\partial M}|\nabla_{\partial M}V|^{2}d\sigma + \varepsilon\int_{M}|\nabla V|^{2}d\sigma}_{\geq0} + 4m\kappa\varepsilon\pi \leq 16\pi - \int_{\partial M}H^{2}d\sigma.
		\end{eqnarray*}
		So,
		\begin{eqnarray*}
			m\leq \frac{1}{\kappa\varepsilon}\left(4 - \frac{1}{4\pi}\int_{\partial M}H^{2}d\sigma\right).
		\end{eqnarray*}
		Combining \eqref{mh} with the last inequality we get the result.
	\end{proof}

	
	\

\end{document}